\documentclass[letterpaper, 10 pt, conference]{ieeeconf} 
\IEEEoverridecommandlockouts                              % This command is only needed if 
\usepackage{mathtools}
\usepackage{amssymb}
\usepackage{bbold}
\usepackage[yyyymmdd,hhmmss]{datetime}
\usepackage{bm,upgreek}
\usepackage{amssymb,amsthm}  
\usepackage{cite}
\usepackage{color}
\usepackage{mathtools}
\mathtoolsset{showonlyrefs,showmanualtags}
\usepackage{dsfont}
\usepackage{booktabs}
\usepackage[final]{graphicx}

\DeclarePairedDelimiter\floor{\lfloor}{\rfloor}
\DeclareMathOperator{\Binopdf}{Binopdf}
\DeclareMathOperator{\VEC}{vec}
\newcommand{\gr}{g^r}
\newcommand{\gd}{g^d}
\newcommand{\Cd}{c_d}
\newcommand{\Cr}{c_r}

\newcommand{\Exp}[1]{\mathbb{E}\left[ #1\right]}
\newcommand{\ID}[1]{ \mathbb{1} \left(#1 \right) }
\newcommand{\Prob}[1]{\mathbb{P} \left(#1\right)}      
 
\newcommand{\Compress}{\medmuskip=0mu
\thinmuskip=0mu
\thickmuskip=0mu}
\usepackage{nth}

\newtheorem{Theorem}{Theorem}

\newtheorem{Remark}{Remark}
\newtheorem{Lemma}{Lemma}
\newtheorem{Corollary}{Corollary}

\newtheorem{Problem}{Problem}

\title{\LARGE \bf
{Optimal  Dynamic Pricing for Binary Demands in Smart Grids: A  Fair  and Privacy-Preserving Strategy$^{\boldsymbol \ast}$}
}

\author{Jalal Arabneydi$^1$ and Amir G. Aghdam$^2$% <-this % stops a space
\thanks{$^\ast$This work has been supported in part by the Natural Sciences and Engineering Research Council of Canada (NSERC) under Grant RGPIN-262127-17, and in part by Concordia University under Horizon Postdoctoral Fellowship.}  
\thanks{$^1$Jalal Arabneydi and $^2$Amir G. Aghdam are with the  Department of Electrical and Computer Engineering, 
        Concordia University, 1455 de Maisonneuve Blvd, Montreal, QC, Canada. 
        {\tt\small Email:jalal.arabneydi@mail.mcgill.ca} and
        {\tt\small Email:aghdam@ece.concordia.ca}}%
}

\begin{document}
\maketitle

\vspace*{-5cm}{\footnotesize{Proceedings of American Control Conference, 2018.}}
\vspace*{4.23cm}

\thispagestyle{empty}
\pagestyle{empty}

%%%%%%%%%%%%%%%%%%%%%%%%%%%%%%%%%%%%%%%%%%%%%%%%%%%%%%%%%%%%%%%%%%%%%%%%%%%%%%%%
\begin{abstract}
Motivated by demand-side management in smart grids, a decentralized controlled  Markov chain formulation is proposed to model a homogeneous population of users with  binary demands (i.e., off or on). The binary demands  often arise in scheduling applications such as   plug-in hybrid vehicles. 
%\B{The  users request electricity   according to an independent Bernoulli probability distribution.}
Normally,  an  independent service operator (ISO) has  a finite number of options  when it comes to providing the users with electricity.  The options  represent  various incentive means,  generation resources, and price profiles.  The objective of the ISO is to find optimal  options in order to keep the distribution of   demands close to a desired level (which varies with time, in general) by imposing the minimum price  on the users.   A Bellman equation is developed here to identify the globally team-optimal strategy. The proposed strategy is fair for all users and also  protects the privacy of  users. Moreover,  its computational complexity increases linearly (rather than exponentially)  with the number of users. A numerical example with 100 users is presented for peak-load management.
\end{abstract}
\section{Introduction}
The power grid is a complex multi-facet network consisting of various interconnected components such as  generators,  transmission lines, and consumers as well as  marketing rules,  etc.  While the power grid plays a key  role in today's quality of life, it increasingly faces serious challenges such as limited fuel-based energy resources, global warming, and the unpredictable nature  of renewable energy sources.  To  adopt to the growing limitations of the future world, the traditional power grid  needs to evolve to a smarter network with more efficient interactions between  its components.  To this end,  there has been a significant surge of interest in research and investment on  smart grids  in recent years~\cite{fang2012smart}.

One of the main challenges  in a smart grid is to  regulate the grid in the face of different types of volatility such as the uncertainty of renewable energy or the continuous variation of  daily load  pattern. One of the most common and effective means of addressing this challenge   is demand-side management, where  (shiftable and deferrable) demands  are intelligently re-distributed  to maintain the aggregate load  close to a desired load pattern.
%For more  details on demand-side management,  the reader is referred to surveys \cite{barbato2014optimization,vardakas2015survey,deng2015survey}  and references therein. 

Several models have been proposed  in  the literature for demand-side management. In  \cite{barbato2014optimization}, three different classifications are introduced for these models. The first one is  deterministic versus stochastic, where stochastic models prove more realistic as they can capture the uncertainty of the grid. The second one is  individual-based versus collaborative, where the collaborative  models  are more complex as they take into account the effect of each demand's decision on other demands. The third one is real-time pricing  versus day-ahead pricing, where day-ahead pricing structure, in addition to the current pricing, considers future pricing as well.   For more  details on  models for demand-side management,  the interested reader is referred to surveys \cite{barbato2014optimization,vardakas2015survey,deng2015survey}  and references therein.

One of the important requirements  in a smart grid is the privacy of information. This means  that  the status  of each  user  should not be shared with  others; consequently, only decentralized strategies are admissible in smart grids.  On the other hand, it is conceptually difficult to find  optimal decentralized strategies because  different users have different perception of   the system as their information is not necessarily the same. To  address this challenge,  the common practice   in the literature is 
to  first find the centralized solution and  then implement it in a distributed manner by using consensus-like algorithms under some convex assumptions.  In this type of approach,  only a weighted average of  the states of  other users is available to each user.  See for more details \cite{samadi2010optimal,dominguez2011distributed,sakurama2017communication} on static models and~\cite{tan2014optimal,chang2012coordinated} on dynamic models. 

In this paper, we propose a controlled Markov chain model for demand-side management for the case where  users have a  binary demand/no-demand state. 
%\R{In this paper, a controlled Markov chain formulation  is proposed to model a homogeneous population of users possessing  binary demand/no-demand state.} 
This model falls into the class of stochastic, collaborative, and day-ahead pricing models. In general, binary demands  emerge  in various applications such as  plug-in electric vehicles~\cite{rotering2011optimal,foster2013optimal}, appliances scheduling~\cite{koutsopoulos2012optimal}, and pool pumps control~\cite{meyn2015ancillary}.  In contrast to the  papers cited in the previous paragrapgh, the dynamics,  price functions  and load pattern trajectory in this paper are described in discrete space (rather than continuous space).  Each user has access only to its local state (private information) and  the empirical distribution of demands (shared information called mean-field). In general, the  computational complexity  of problems with a non-classical information structure  described above  is NEXP~\cite{Bernstein2002complexity}.   Inspired by mean-field teams approach~\cite{arabneydi2016new,JalalCDC2014,JalalCDC2017}, we develop a Bellman equation to identify the globally  team-optimal solution.   It is shown that the computational complexity of the proposed solution increases linearly (rather than exponentially) with the number of users.

This paper is organized as follows. In Section~\ref{sec:problem}, the problem is formulated and  the main  challenges of finding the optimal solution are outlined. In Section~\ref{sec:theoretical},  a Bellman equation is derived to identify the optimal solution. In Section~\ref{sec:numerical},  the efficiency of the proposed solution is demonstrated by  numerical example  and in Section~\ref{sec:concolution}, the paper is concluded.

\subsection{Notation}
$\mathbb{N}$  and $\mathbb{R}$ refer to natural numbers and real numbers, respectively.  Given $k \in \mathbb{N}$, $\mathbb{N}_k$ denotes the finite set  $\{1,\ldots,k\}$ and  $x_{1:k}$ denotes vector $\VEC(x_1,\ldots,x_k)$. The probability, expectation, and indicator functions of an event  are respectively  denoted by $\Prob{\cdot}$,  $\Exp{\cdot}$, and  $\mathbb{1}(\cdot)$.  For any $n \in \mathbb{N}$ and $p \in [0,1]$,   $\Binopdf(\boldsymbol \cdot ,p,n )$ denotes the binomial probability distribution of $n$ trails with success probability~$p$.   

\section{Problem formulation}\label{sec:problem}
Consider an \emph{Independent Service Operator (ISO)}  regulating the power grid  by managing the demand side. Let $n \in \mathbb{N}$ denote the number of  users (consumers) in the grid.  At every time instant, each user requests power (makes a demand)  with probability $p \in (0,1)$. The demand processes of users are assumed to be mutually  independent. It is to be noted that  ISO does not necessarily serve (deliver) each demand immediately. In such a case, the corresponding user is not allowed to make a new demand  until after the existing  one is delivered by ISO.  Denote by $x^i_t \in \mathcal{X}:=\{0,1\}$ the state of user~$i \in \mathbb{N}_n$  at time~$t \in \mathbb{N}$,  where $x^i_t=1$ implies user~$i$ has a demand at time~$t$. Let $m_t \in \mathcal{M}:=\{0, \frac{1}{n}, \frac{2}{n},\ldots,1\}$ denote the empirical distribution of demands at time $t$, i.e., 
\begin{equation}\label{eq:mf}
m_t=\frac{1}{n}\sum_{i=1}^n \ID{x^i_t=1}.
\end{equation}
For ease of reference,  the empirical distribution of demands will hereafter be called \emph{mean-field}.

\subsection{Model}
Suppose ISO has $k \in \mathbb{N}$  options to provide  power  to  users, where each option has its unique  price,  participation, and delivery rate. The options may represent different price profiles,  incentives,  and generation resources.  Define the following for  every option $u \in \mathbb{N}_k$.

\emph{1) Participation rate}  is the probability according to which  a user is incentivized to delay its demand at any  time instant $t \in \mathbb{N}$. The ISO offers monetary discounts to users to incentivize  them to delay their demands.   Denote  the participation rate\footnote{In practice, the participation rate $\alpha(u)$ may depend on the monetary discount and the monetary discount  may   indirectly depend on $m_t$. As a result, the participation rate may be a function of $m_t$,  and hence be denoted as $\alpha(u,m_t)$.  For simplicity of presentation,  this indirect dependence is ignored in this paper; however, our results hold for  such  case as well.} by $\alpha(u) \in [0,1]$, and assume that it  is  independent of the demands. Hence, the probability of a demand being made by a participant user  is $ (1-\alpha(u)) p \leq p$. 

\emph{ 2) Delivery rate}  is the probability according to which  the demand of a user  is delivered  at any time $t \in \mathbb{N}$.  Since ISO often uses multiple   sources of energy  with different time responses and capacities,  it is possible that the rate of  delivery changes based on the type of power generation  source.  Denote the delivery rate  by $q(u, m_t) \in (0,1]$. Since the delivery rate may decrease as the the number of demands increases,  we use term $m_t$ to  denote this dependence.

\emph{3) Reserve price} is the price that a user with no demand $(i.e., x=0)$ pays for reliable energy (i.e., for the maximum power guaranteed by the ISO). At any time instant $t \in \mathbb{N}$,  ISO  reserves energy for the next time instant to reduce the chances of a power outage (e.g., in case of an upcoming unexpectedly high demand).  Each user must pay this price even if  it does not use  the electricity.   Denote the reserve price by~$\Cr(u, 1-m_t) \in \mathbb{R}_{\geq 0}$.\footnote{Since the reserve  price  may vary  based on the size of capacity (i.e., $n(1-m_t)$), we use term $(1-m_t)$ to denote this dependence.}

\emph{4) Demand price}  is the price that a user with demand $(i.e., x=1)$ pays  for energy. This includes the price of energy as well as  the price of the reserved capacity of the next time instant. It is to be noted that  users with demand can use an option different from  users with no demand. This flexibility enables ISO to manage its resources more efficiently. Denote the demand price by $\Cd(u, m_t) \in \mathbb{R}_{\geq 0}$.\footnote{Similar to the reserve price,  the demand price  may vary  based on the size (i.e., $n m_t $); hence,  we use term $m_t$ to denote this dependence.}  
\begin{Remark}
\emph{Notice  that $\alpha(u)=1$  and    $q(u,m_t)=1$ mean full participation and  instant delivery, respectively. When the reserve price  $\Cr(u,1-m_t)$  and  demand price  $\Cd(u,m_t)$  are independent of $m_t$, the price  is said to be flat.}
\end{Remark}
Let  $u^i_t \in \mathbb{N}_k$ denote  the option assigned to user $i \in \mathbb{N}_n$ at time $t \in \mathbb{N}$.   Then, given $m_t \in \mathcal{M}$, the transition probability matrix of user $i $  at time $t $ can be written as follows:
\begin{equation}\label{eq:dynamics}
\Prob{x^i_{t+1} | x^i_t, u^i_t,m_t}= \left[ 
\begin{array}{c c}
1- (1-\alpha(u^i_t)) p & (1- \alpha(u^i_t))p\\
q(u^i_t,m_t) & 1- q(u^i_t,m_t)
\end{array}
\right].
\end{equation}
In addition,  the price that user $i$ pays at time $t$  can be expressed as:
\begin{multline}\label{eq:per-step}
c(x^i_t,u^i_t,m_t)=\ID{x^i_t=0}\Cr(u^i_t,1-m_t)\\+\ID{x^i_t=1}\Cd(u^i_t,m_t).
\end{multline}

\subsection{Strategy}
 Any admissible strategy employed by the  ISO must satisfy the following two conditions:
\begin{itemize}
\item They must be ``fair" in the sense that no user is privileged or  discriminated in the grid.

\item  They must be designed  in such a way that  the  ``privacy"  of users is protected  so that no third-party or ISO can know the  status of each user.
\end{itemize}
Inspired by the above  features, we consider the following control law:
\begin{equation}\label{eq:info-structure}
u^i_t=g_t(x^i_t,m_{1:t}),
\end{equation}
where $g_t: \mathcal{X} \times \mathcal{M}^t \rightarrow \mathbb{N}_k$, $t \in \mathbb{N}$. The collection of control laws $g_t$ is called the 
\emph{strategy} of system, i.e., $g:=\{g_1,g_2, \ldots\}$. This strategy is fair because if two  users have the same state, their control actions will be identical. In addition,   under this strategy,  the state of user $i$, i.e. $x^i_t$, is  known to  neither ISO nor other users.

Each user is equipped with a smart meter which enables the user to  control its demand and communicate with  the ISO and other users.   There are several ways to share  mean-field $m_t$ among users.  For instance, ISO  can compute the mean-field  by collecting the states of users in an anonymous fashion and then  broadcast the mean-field to all users. Another way is  to run a  proper consensus algorithm, e.g.,~ \cite{Olfati2006belief,Bishop2014},  in the network in  order to share the mean-field among users in a distributed manner, independently of  ISO. 

\begin{Remark}
\emph{When the number of users is asymptotically large, the mean-field becomes deterministic, hence predictable. In such case,  there is no need to even share  the mean-field~\cite{JalalCDC2017}.}
\end{Remark}

\subsection{Optimization problem}
Let $\theta_t \in [0,1]$ denote the trajectory of  the desired  load pattern at time $t \in \mathbb{N}$, and  $s_t \in \mathcal{S}$ denote  its state space representation such that
\begin{align}\label{eq:theta}
s_{t+1}&=f(s_t), \nonumber\\
\theta_t&=h(s_t),
\end{align}
where $f: \mathcal{S} \rightarrow \mathcal{S}$  and $h: \mathcal{S} \rightarrow [0,1]$.   The ISO provides  users with   initial state $s_1$ and functions $f$ and $h$.

The ISO ideally looks for the optimal option  to keep  the distribution of demands close to a desired load trajectory by imposing minimum price on users.  Given discount factor $\beta \in (0,1)$,  define the following cost function:
\begin{equation}\label{eq:total cost}
J(g)= \mathbb{E}^g\left[\sum_{t=1}^\infty \beta^{t-1} \frac{1}{n}\left( \sum_{i=1}^n c(x^i_t,u^i_t,m_t) + D(m_t,\theta_t) \right) \right],
\end{equation}
where $D: \mathcal{M} \times [0,1] \rightarrow \mathbb{R}_{\geq 0}$ denotes  an arbitrary distance function.  Note that the above expectation function depends on the  choice of strategy~$g$. 
\begin{Problem}\label{prob}
Find the optimal strategy~$g^\ast$ such that for every strategy $g$,
\begin{equation}
J(g^\ast) \leq J(g).
\end{equation}
\end{Problem}

Since the information structure, given by \eqref{eq:info-structure},  is non-classical, 
% it is conceptually  challenging to find a solution  for Problem~\ref{prob}.   In particular, 
  every user has different information (perception) about the system and such discrepancy  makes it difficult to establish cooperation among users.  In general, the computational complexity of  Problem~\ref{prob}  increases exponentially with  the number of users and  control horizon~\cite{Bernstein2002complexity},  which  makes Problem~\ref{prob} computationally expensive  even  for a relatively small number of users.   This complexity is exasperated  when the control horizon is infinity.
%  In addition, when attention is restricted to fair strategies, randomized strategies may  outperform non-randomized (pure) strategies~\cite{}; consequently,  the domain of optimal strategy  extends  to the space of probability measures over the domain of strategy~\eqref{eq:info-structure}.
 
%We develop a  Bellman equation to identify the optimal solution for Problem~\ref{prob}. We show that the state space of this Bellman equation  increases \emph{linearly} (rather than exponentially) with the number of users.    As a result,  the optimal solution can be computed  for applications with the moderate  number of users.
\section{Main results}\label{sec:theoretical}
In this section, we present a solution to Problem~\ref{prob}. Given  realization $m_{1:t} \in \mathcal{M}^t, t \in \mathbb{N}$,  define \emph{reserve action} $\gr_t \in \mathbb{N}_k$  for users with no demand as follows:
\begin{equation}\label{eq:gr}
\gr_t:= g_t(0, m_{1:t}).
\end{equation}
Similarly, for users with demand define  \emph{demand action}  $\gd_t \in  \mathbb{N}_k$ as:
\begin{equation}\label{eq:gd}
\gd_t:= g_t(1, m_{1:t}).
\end{equation}
Equations \eqref{eq:info-structure}, \eqref{eq:gr}, and \eqref{eq:gd}  yield
\begin{equation}\label{eq:g-u}
u^i_t=\ID{x^i_t=0}\gr_t +\ID{x^i_t=1}\gd_t.
\end{equation}
We show that the stochastic process $m_{1:t}$ evolves in a Markovian manner under the reserve and demand actions.  This implies there is no loss of optimality in dropping $m_{1:t-1}$.
\begin{Theorem}\label{thm:markov-mf}
For any $m_{1:t} \in \mathcal{M}^t$,  $ \gr_{1:t}   \in (\mathbb{N}_k)^t$, and $\gd_{1:t} \in (\mathbb{N}_k)^t$, $t \in \mathbb{N}$, the following equality holds:
\begin{equation}
\Prob{m_{t+1}| m_{1:t}, \gr_{1:t}  ,\gd_{1:t} }=\Prob{m_{t+1}| m_t, \gr_t,\gd_t}.
\end{equation}
\end{Theorem}
\begin{proof}
From~\eqref{eq:mf},
\begin{equation}\label{eq:proof-next-mf}
n\cdot m_{t+1}=\sum_{i=1}^n \ID{x^i_{t+1}=1}.
\end{equation}
Suppose every agent has access to  $\{x^1_{1:t}, \ldots, x^n_{1:t}\}$. Then,  $\{\ID{x^i_{t+1}=1}\}_{ i \in \mathbb{N}_n}$ are independent  Bernoulli random variables  with identical probability distribution according to \eqref{eq:dynamics}, i.e., $\Prob{x^i_{t+1}=1 \mid x^i_t, u^i_t, m_t}$ where  $u^i_t$ and $m_t$ are given by \eqref{eq:g-u} and \eqref{eq:mf}, respectively.  Since   \eqref{eq:proof-next-mf} is  a summation of $n$ independent random  variables, its   probability distribution  can be expressed as the convolution of  the probability distributions of  these $n$ random variables~\cite{grinstead2012introduction}. In particular,  $n(1-m_t)$ of the  random variables  have   state zero at time~$t$, i.e. $x^i_t=0$, and the remaining  $n m_t$  of them  have state 1, i.e. $x^i_t=1$.  Therefore,   the following yields from \eqref{eq:proof-next-mf}:
\begin{multline}\label{eq:proof-x-m}
\Prob{n \cdot m_{t+1} \mid x^1_{1:t}, \ldots, x^n_{1:t}, \gr_{1:t},\gd_{1:t} }=  \\  \underbrace{\Prob{1 \mid 0, \gr_t, m_t} \ast \ldots \ast \Prob{1 \mid 0, \gr_t, m_t}}_{n(1-m_t)} \\
 \ast\underbrace{\Prob{1 \mid 1, \gd_t, m_t} \ast \ldots \ast \Prob{1 \mid 1, \gd_t, m_t}}_{nm_t}.
\end{multline}
On the one hand, the history of mean-field $m_{1:t}$ is a function of the history of joint state $x^1_{1:t}, \ldots, x^n_{1:t}$  according to  \eqref{eq:mf}, and on the other hand, the right-hand side of \eqref{eq:proof-x-m} depends  only on $m_t, \gr_t,$ and $\gd_t$. Thus,  it follows from  \eqref{eq:proof-x-m}  that
\begin{equation}
\Prob{m_{t+1}| m_{1:t}, \gr_{1:t}  ,\gd_{1:t} }=\Prob{m_{t+1}| m_t, \gr_t,\gd_t}.
\end{equation}
\end{proof}
\begin{Remark}
\emph{It is to be noted that the mean-field process $ m_{1:t}$ is not necessarily a controlled Markov process under joint actions $\{ u^1_{1:t}, \ldots,u^n_{1:t} \}$, i.e.,
\begin{equation}
\Prob{m_{t+1}| m_{1:t},  u^1_{1:t}, \ldots,u^n_{1:t}} \neq \Prob{m_{t+1}| m_t,  u^1_{t}, \ldots,u^n_{t}}.
\end{equation}
% A counterexample is as follows:  let $x^i_{t+1}=x^i_t \cdot \ID{u^i_t=1}$, $i \in \mathbb{N}\backslash\{1,2\}$, where $u^i_t \in \{1,2\}$. Then, $m_{t+1}$ cannot be expressed as a function of  $m_{t}$ and $\{ u^1_{t}, \ldots,u^n_{t} \}$; however,   it can be written as a function of $m_{t}$ and $\gd_t$, i.e.  $m_{t+1}= m_t \ID{\gd_t=1}$.
 }
\end{Remark}
\begin{Theorem}\label{thm:bino}
The transition probability matrix of the mean-field can be expressed  in terms of binomial probability distributions as follows. Given any $\hat m \in \mathcal{M}$,
\begin{itemize}
\item for $m_t=0$, 
\begin{equation}
\Compress
\Prob{m_{t+1}=\hat m| m_t, \gr_t,\gd_t}=\Binopdf\big(n\hat m, (1-\alpha(\gr_t))p , n\big),
\end{equation}
\item for $m_t=1$, 
\begin{equation}
\Compress
\Prob{m_{t+1}=\hat m| m_t, \gr_t,\gd_t}=\Binopdf\big(n\hat m, 1-q(\gd_t,m_t) , n \big),
\end{equation}
\item  for $m_t \notin \{0,1\} $, 
\begin{align}
&\Prob{m_{t+1}=\hat m| m_t, \gr_t,\gd_t}=\\
& \Big(\Binopdf\big(\boldsymbol \cdot, (1-\alpha(\gr_t))p , n (1-m_t)\big) \\
& \hspace{1.6cm}\ast \Binopdf\big(\boldsymbol \cdot, 1-q(\gd_t,m_t) , n m_t\big)\Big)(n\hat m).
\end{align}
\end{itemize}
\end{Theorem}
\begin{proof}
The proof follows from the interesting connection between convolution power and binomial distribution~\cite[chapter 7.1]{grinstead2012introduction}.  The first group of convolutions in  \eqref{eq:proof-x-m}  can be interpreted as the probability distribution of $n(1-m_t)$  Bernoulli random variables (trials) with success probability $\Prob{1 \mid 0, \gr_t,m_t}=(1-\alpha(\gr_t))p$. Similarly, the second group of convolutions in \eqref{eq:proof-x-m} can be interpreted as the probability distribution of $n m_t$ trials with success probability $\Prob{1 \mid 1, \gd_t,m_t}= 1-q(\gd_t,m_t)$. 
\end{proof}
%Let  $z_t=(m_t,s_t) \in \mathcal{Z}:=\mathcal{M} \times \mathcal{S}, t \in \mathbb{N}$.
\begin{Lemma}\label{lemma:per-step}
For any $m_t \in \mathcal{M}$, $s_t \in \mathcal{S}$, $\gr_t \in \mathbb{N}_k$, and $\gd_t \in \mathbb{N}_k$,  $t \in \mathbb{N}$,  there exists a function  $\ell$ such that
\begin{equation}
\ell(m_t,s_t, \gr_t,\gd_t)=\frac{1}{n}\left( \sum_{i=1}^n c(x^i_t,u^i_t,m_t) + D(m_t,\theta_t) \right).
\end{equation} 
\end{Lemma}
\begin{proof}
The proof follows from \eqref{eq:per-step}, \eqref{eq:theta} and \eqref{eq:g-u} as shown below
\begin{align*}
&\frac{1}{n}\left( \sum_{i=1}^n c(x^i_t,u^i_t,m_t) + D(m_t,\theta_t) \right)\\
&=\sum_{x \in \mathcal{X}} \left(  \frac{1}{n}  \sum_{i=1}^n \ID{x^i_t=x}c(x^i_t,u^i_t,m_t) \right)+ D(m_t,\theta_t)\\
&=(1-m_t)\Cr(\gr_t,1-m_t)+m_t \Cd(\gd_t,m_t)+D(m_t,h(s_t))\\
&=:\ell(m_t,s_t, \gr_t,\gd_t).
\end{align*}
\end{proof}
\begin{Theorem}\label{thm:bellman}
For any $m \in \mathcal{M}$ and $s \in \mathcal{S}$,  define the following Bellman equation
\begin{multline}\label{eq:value-inf}
V(m, s):=\min_{\gr \in \mathbb{N}_k, \gd \in \mathbb{N}_k}\Big(\ell(m,s,\gr,\gd)+\\   \beta \sum_{\hat m \in \mathcal{M}} \Prob{\hat m|m,\gr,\gd} V(\hat m, f(s))\Big).
\end{multline}
Let $(\mathfrak{g}^{r,\ast}(m,s),\mathfrak{g}^{d,\ast}(m,s))$ denote a minimizer of \eqref{eq:value-inf}.  Then, the optimal control law for Problem~\ref{prob} is given by
\begin{equation}\label{eq:optimal_action}
u^{i,\ast}_t=\ID{x^i_t=0}\mathfrak{g}^{r,\ast}(m_t,s_t)+\ID{x^i_t=1}\mathfrak{g}^{d,\ast}(m_t,s_t).
\end{equation}
\end{Theorem}
\begin{proof}
 For any $T\in \mathbb{N}$,  define the finite-horizon counterpart of $J(g)$, given by \eqref{eq:total cost}, as follows:
\begin{equation}\label{eq:proof-finite-cost}
J_T(g):= \mathbb{E}^g\left[\sum_{t=1}^T \beta^{t-1} \frac{1}{n}\left( \sum_{i=1}^n c(x^i_t,u^i_t,m_t) + D(m_t,\theta_t) \right) \right].
\end{equation}
From \eqref{eq:theta} and Theorem~\ref{thm:markov-mf},
\begin{multline}
\Prob{m_{t+1},s_{t+1} \mid m_{1:t},s_{1:t},\gr_{1:t},\gd_{1:t}}\\ =\Prob{m_{t+1}|m_t,\gr_t,\gd_t}\ID{s_{t+1}=f(s_t)}.
\end{multline}
Therefore, $(m_{1:t},s_{1:t})$  is a controlled Markov process under reserve and demand  actions $(\gr_{1:t},\gd_{1:t})$. In addition, from Lemma~\ref{lemma:per-step}, the per-step is a function of $(m_t,s_t,\gr_t,\gd_t)$. It is well-known in Markov decision theory \cite{kumar2015stochastic} that in this case there exists a dynamic programming decomposition  to identify  a minimizer to  \eqref{eq:proof-finite-cost}, i.e., for $t \in \mathbb{N}_T$,
\begin{multline}
V_{t}(m,s):=\min_{\gr \in \mathbb{N}_k, \gd \in \mathbb{N}_k} \big( \ell(m,s,\gr,\gd)\\ +\sum_{\hat m \in \mathcal{M}} \Prob{\hat m|m,\gr,\gd} V_{t+1}(\hat m, f(s))\big),
\end{multline}
where $V_{T+1}(m,s):=0, \forall m \in \mathcal{M}, \forall s \in \mathcal{S}$.  Using the proof technique of \cite{kumar2015stochastic}, we make the following change of variable   for  any $m \in \mathcal{M},  s \in \mathcal{S}$ and any $ t \in \mathbb{N}_T$,
\begin{align}\label{eq:def-W-infinite}
W_t( m,s):=\beta^{-T+t-1} V_{T-t+2}( m,s),
\end{align}
where
$W_1( m,s)=\beta^{-T} V_{T+1}(m,s)=0.
$
Using simple algebraic manipulations, we arrive at
\begin{multline}
W_{T+1}(m,s)=\min_{gr \in \mathbb{N}_k, \gd \in \mathbb{N}_k} \big( 
\ell(m,s,\gr,\gd)+\\   \beta \sum_{\hat m \in \mathcal{M}} \Prob{\hat m|m,\gr,\gd} W_T( \hat m, f(s))
\big).
\end{multline}
Since the above Bellman operator is contractive~\cite{kumar2015stochastic},  its solution converges to a fixed-point solution as time goes to infinity, i.e.,
\begin{equation}\label{eq:V-infinite-mfs}
\lim_{T \rightarrow \infty} W_{T} =  W_{\infty}=:V. 
\end{equation}
\end{proof}
Since the mean-field and desired trajectory of demands  are known to all users, every user can independently solve \eqref{eq:value-inf} and find the optimal strategy in a distributed manner\footnote{There are various ways to  solve the Bellman equation~\eqref{eq:value-inf}  such as  value iteration and policy iteration~\cite{kumar2015stochastic}. }.  Each user constructs  its optimal  action~\eqref{eq:optimal_action} based on the obtained strategy,  private information $x^i_t$,  mean-field $m_t$, and the state of the desired trajectory, i.e. $s_t$.  As a result, each user will only need to know the mean-field rather than  the local information of other users. 
\begin{Remark}
\emph{  It is important to note that the cardinality of set $\mathcal{M}$ increases linearly with the number of users $n$ while  the cardinality of set $\mathcal{S}$ is independent of $n$. Therefore,   the computational complexity of  the solution of the Bellman equation~\eqref{eq:value-inf} increases linearly with~$n$, given the transition probability matrix\footnote{Binomial probability distribution function can be efficiently computed by different methods such as discrete Fourier transformation.} in Theorem~\ref{thm:bino}.  Note also that the computational  complexity of the Bellman equation~\eqref{eq:value-inf} increases exponentially with time, in general; however,  it is common practice to use  value iteration  to  find an $\varepsilon$-optimal solution in polynomial time,  where $\varepsilon$ converges  to zero exponentially~\cite{Bertsekas2012book}.}
\end{Remark}

According to the proof of Theorem~\ref{thm:bellman}, the pair $(m_t,s_t)$ evolves in a Markovian manner under reserve and demand actions. Hence, the Bellman equation of Theorem~\ref{thm:bellman} extends to the following case.
\begin{Corollary}\label{cor}
Let the dynamics  \eqref{eq:dynamics} and per-step costs \eqref{eq:per-step}  depend on the state of the desired load trajectory $s_t, t \in \mathbb{N}$. Then, the result of Theorem~\ref{thm:bellman} still holds, and for any $m \in \mathcal{M}$ and $s \in \mathcal{S}$,  the Bellman equation is given by
\begin{multline}
V(m, s):=\min_{\gr \in \mathbb{N}_k, \gd \in \mathbb{N}_k}\Big(\ell(m,s,\gr,\gd)+\\   \beta \sum_{\hat m \in \mathcal{M}} \Prob{\hat m|m,s,\gr,\gd} V(\hat m, f(s))\Big).
\end{multline}
\end{Corollary}

\section{Numerical Results}\label{sec:numerical}
%In this section, we present two numerical examples where in the first example the desired reference trajectory of the empirical distribution of demands is constant and  in the second one it is variant.  
%Power shortage  during the peak-load period  is an important problem in a smart grid.
%  In such situation,  two scenarios may occur.  In the first one,  the grid has to use  expensive  back-up (ancillary) generators to compensate the shortage  or  in the second case,   the grid has to  experience blackouts. In either case,  the cost is very high.
 Motivated by  peak-load management,  we  present a numerical example to illustrate our results. 

\textbf{Example 1.}  Denote by $T \in \mathbb{N}$ the number of equal control intervals within a full day ($24$ hours), and suppose that the ISO applies the  control strategy at the beginning of each interval, i.e.,   it  intervenes   every $\frac{24 \times 60}{T}$ minutes.
% Denote by  $T \in \mathbb{N}$  the number of control-time samples in one full day, i.e., the control sample rate is $\frac{24 \times 60}{T}$ minutes. 
For any time  $t \in \mathbb{N}$,   define \emph{daily control time instants}  by the following modulo-operation:
\begin{equation}
s_t:=\begin{cases}
t- \floor{\frac{t}{T}}T, & \frac{t}{T} \neq \floor{\frac{t}{T}},\\
T, & \frac{t}{T} = \floor{\frac{t}{T}}.
\end{cases}
\end{equation}
 For example, let $T=100$ and $t=240$, then $s_{240}$ refers to the control interval $40$ in day $3$.
From the above definition,  $s_t \in \mathbb{N}_T, t \in \mathbb{N}, $ evolves as follows:
\begin{equation}
s_{t+1}=\ID{s_t <T} (s_t+1) +\ID{s_t=T}, 
\end{equation} 
where  the initial value $s_1=1$.
Let $\tau_B < T$ and $\tau_E < T$ denote respectively  the  first and  last  daily control  time instants of  the peak-load period.  The desired load pattern trajectory is valley-shaped  as follows:
\begin{equation}
\theta_{t}=\begin{cases}
0.8, &  1 \leq s_t < \tau_B,\\
0.8 - 0.6 sin(\dfrac{s_t - \tau_B}{\tau_E-\tau_B} \pi), & \tau_B \leq  s_t < \tau_E,\\
0.8, &    \tau_E \leq s_t \leq T.
\end{cases}
\end{equation} 
 Suppose  the ISO has the following three options ($ k=3$) at each time $t \in \mathbb{N}$.\footnote{The options can be extended to  multiple  electricity-providers with various price, incentive, and delivery rates.}

\textbf{Basic option}: This  option, denoted by $u=1$, refers to the case where the electricity is provided by the main generators. In this option:
\begin{itemize}
\item There is no incentive for users to delay their demands, i.e., participation rate $\alpha(1)$ is $0$.

\item  The  demands are served at rate $q(1,m_t)=0.2$.
%\footnote{When the number of demands is small, the  demands are served quickly.}

\item The per-user price may  increase with respect to the number of served users.  Let the price functions  be $\Cr(1,1-m_t)=1+ (1-m_t)$ and $\Cd(1,m_t)=1.5+1.5 m_t$.
%\footnote{For simplicity, we assume the price  of electricity in off-peak and on-peak  periods are the same. However, one way to handle this is through the dependence on $m_t$, e.g., $\Cr(1,1-m_t)=0.1 +0.2\ID{m_t \geq 0.6}$ and $\Cd(1,m_t)=1+2\ID{m_t \geq 0.6}$}
\end{itemize}

\textbf{Ancillary option}: This  option, denoted by $u=2$, refers to the case where  ancillary (back-up) generation and reserve resources are  integrated  to the grid. In this option:
\begin{itemize}
\item There is no incentive for users to delay their demands, i.e., participation rate $\alpha(2)$ is $0$.

\item The  demands are served at a higher rate than that of basic option. Let  $q(2,m_t)=0.4$. 

\item  Let the price of electricity  be  flat as follows: $\Cr(2,1-m_t)=2$ and $\Cd(2,m_t)=3$.
\end{itemize}

\textbf{Incentive-based option}: This  option, denoted by $u=3$, refers to the case where ISO offers monetary discounts to users for incentivizing them to cooperate with ISO. If a user has no demand at time $t$,  it receives  $0.05$ discount unit  as an incentive  for not making a demand at $t+1$. If a user has a demand  at time $t$,  in addition to  $0.05$ unit discount for  not sending a demand at $t+1$, it receives an extra $0.05$ unit discount compared to the basic option for  letting the ISO serve its current demand with more delay (lower delivery rate). In this option: 
\begin{itemize}
\item The participation rate $\alpha(3)$ is assumed to be $0.85$.

\item Let the delivery rate be given by  $q(3,m_t)=0.15 \leq q(1,m_t), m_t \in \mathcal{M}$. 

\item  The discounted prices are assumed to be $\Cr(3,1-m_t)=\Cr(1,1-m_t)-0.05$ and $\Cd(3,m_t)=\Cd(1,m_t)-0.1$.
\end{itemize}
%\subsection{Time-invariant reference}
%Let $s_{t+1}=s_t=0.8$ and $\theta_t=s_t$. Then the optimal solution is displayed in Figure~\ref{}.
%
\begin{figure}[t!]
\centering
\includegraphics[ trim={0cm 5.5cm 0 4.5cm},clip,width=\linewidth]{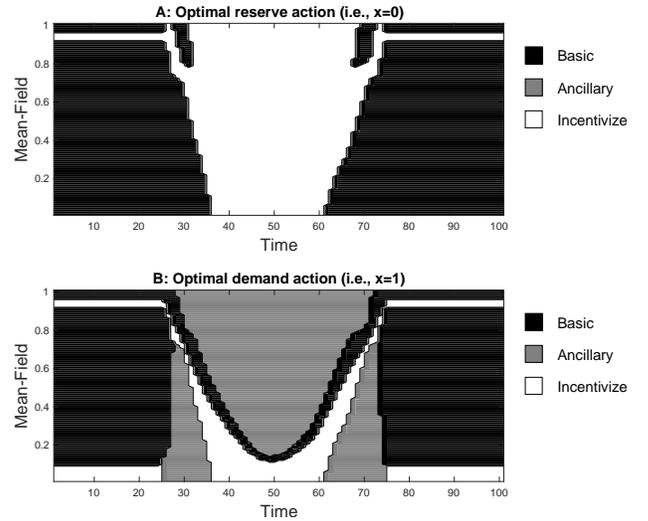}
\caption{The optimal strategy is shown as a function of local state $x$, mean-field $m$, and daily control time instants  $s$ in Example 1.  The peak-load period occurs between  daily control time instants  $25$ and $75$.  The three options (basic, ancillary, incentive-based) are depicted by three different colors.}\label{fig2}
\end{figure}

Consider the following parameters: number of users $n=100$,  demand probability $p=0.8$, discount factor $\beta=0.9$, control  intervals $\tau_B=25$, $\tau_E=75 $, and scaled Euclidean distance function $D(m,\theta)=100|m-\theta|, m \in \mathcal{M}, \theta \in [0,1]$.  The ISO is interested to keep the distribution of demands close to the desired trajectory by imposing minimum price on users. From Theorem~\ref{thm:bellman},   the optimal strategy  can be obtained by solving the  Bellman equation~\eqref{eq:value-inf}.  According to \eqref{eq:optimal_action}, the optimal action of user $i \in \mathbb{N}_n$ at time $t \in \mathbb{N}$ depends on state $x^i_t$, mean-field $m_t$, and the state of the desired trajectory~$s_t$ (i.e., daily control time instants).  The optimal strategy is displayed in Figure~\ref{fig2}.
\begin{figure}[h]
\centering
\scalebox{1.05}{
\includegraphics[trim={2cm 5cm 0 5cm},clip, width=\linewidth]{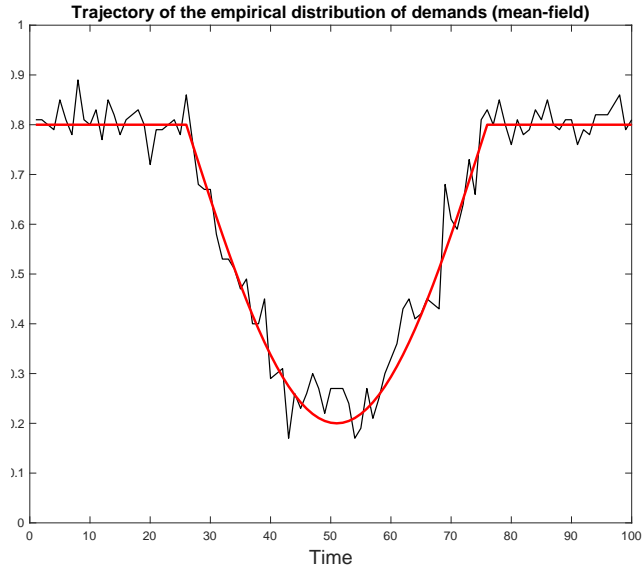}}
\caption{ The  ISO aims  to keep the distribution of demands (mean-field) in Example 1 close to a  valley-shaped trajectory shown in red.  A sample-path of the mean-field under the optimal  strategy is displayed in black. }\label{fig3}
\end{figure}

%\begin{figure}[h]
%\centering
%\scalebox{1}{
%\includegraphics[trim={2cm 3cm 1cm 1cm},clip, width=\linewidth]{value200_best}}
%\caption{The optimal cost  is the value function given by Theorem~\ref{thm:bellman}.}\label{fig1}
%\end{figure}
In Figure~\ref{fig3}, a sample-path (trajectory) of mean-field (empirical distribution of demands) is depicted.  Prior to the peak-load time,  the number of demands is around $80$ and ISO uses  option $1$ (basic option) for both reserve and demand actions according to Figure~\ref{fig2}. In the peak-load time,  the ISO mostly  uses option $3$ (incentive-based  option) for users that have no demand at the moment to encourage them not to make a demand in the future. However, for users that have already requested power (i.e.,  have an active demand at the moment),  the ISO  uses option $2$ (ancillary option) in combination with the other two options in order to keep the distribution of demands close to the desired trajectory with the minimum price.

%To the best of our knowledge, there is currently no approach in the literature to solve the peak-load management problem investigated in Example 1, as the existing algorithms for this type of problem are exponential time, in general. However, our proposed approach could handle the problem as its computational complexity is linear.

%\begin{Remark}
%\emph{It is worth highlighting that  the  above options may  be replaced by multiple  electricity-providers with various price, incentive, and delivery rates.  In addition,   the rates can depend on time-sample  $s_t, t \in \mathbb{N},$ according to Corollary~\ref{cor}.}
%% In such case, the  optimal solution would be more complicated than the one presented above.
%\end{Remark}
\section{Conclusions}\label{sec:concolution}
We proposed a stochastic model for binary demands in  demand-side management. Unlike many existing models,  here the state and action spaces  are  discrete  and  the dynamics and cost  are not necessarily  smooth or convex; hence,  traditional consensus-like approaches may not be as effective. The information structure is non-classical, and the computational complexity of finding a solution increases exponentially with the number of users, in general. We identified  an optimal solution by the Bellman equation. Since the state space of the Bellman equation increases  linearly with the number of users, the computational complexity of the proposed solution is linear with respect to the number of users. We then presented a numerical example for peak-load management to support our results.
\bibliography{MFT_Ref}

\begin{thebibliography}{10}
\providecommand{\url}[1]{#1}
\csname url@samestyle\endcsname
\providecommand{\newblock}{\relax}
\providecommand{\bibinfo}[2]{#2}
\providecommand{\BIBentrySTDinterwordspacing}{\spaceskip=0pt\relax}
\providecommand{\BIBentryALTinterwordstretchfactor}{4}
\providecommand{\BIBentryALTinterwordspacing}{\spaceskip=\fontdimen2\font plus
\BIBentryALTinterwordstretchfactor\fontdimen3\font minus
  \fontdimen4\font\relax}
\providecommand{\BIBforeignlanguage}[2]{{%
\expandafter\ifx\csname l@#1\endcsname\relax
\typeout{** WARNING: IEEEtran.bst: No hyphenation pattern has been}%
\typeout{** loaded for the language `#1'. Using the pattern for}%
\typeout{** the default language instead.}%
\else
\language=\csname l@#1\endcsname
\fi
#2}}
\providecommand{\BIBdecl}{\relax}
\BIBdecl

\bibitem{fang2012smart}
X.~Fang, S.~Misra, G.~Xue, and D.~Yang, ``Smart grid—the new and improved
  power grid: {A} survey,'' \emph{IEEE communications surveys \& tutorials},
  vol.~14, no.~4, pp. 944--980, 2012.

\bibitem{barbato2014optimization}
A.~Barbato and A.~Capone, ``Optimization models and methods for demand-side
  management of residential users: {A} survey,'' \emph{Multidisciplinary
  Digital Publishing Institute, Energies}, vol.~7, no.~9, pp. 5787--5824, 2014.

\bibitem{vardakas2015survey}
J.~S. Vardakas, N.~Zorba, and C.~V. Verikoukis, ``A survey on demand response
  programs in smart grids: Pricing methods and optimization algorithms,''
  \emph{IEEE Communications Surveys \& Tutorials}, vol.~17, no.~1, pp.
  152--178, 2015.

\bibitem{deng2015survey}
R.~Deng, Z.~Yang, M.-Y. Chow, and J.~Chen, ``A survey on demand response in
  smart grids: Mathematical models and approaches,'' \emph{IEEE Transactions on
  Industrial Informatics}, vol.~11, no.~3, pp. 570--582, 2015.

\bibitem{samadi2010optimal}
P.~Samadi, A.-H. Mohsenian-Rad, R.~Schober, V.~W. Wong, and J.~Jatskevich,
  ``Optimal real-time pricing algorithm based on utility maximization for smart
  grid,'' \emph{in Proceedings of the 1st IEEE International Conference on
  Smart Grid Communications}, pp. 415--420, 2010.

\bibitem{dominguez2011distributed}
A.~D. Dominguez-Garcia and C.~N. Hadjicostis, ``Distributed algorithms for
  control of demand response and distributed energy resources,'' \emph{in
  Proceedings of 50th IEEE Decision and Control and European Control Conference
  (CDC-ECC)}, pp. 27--32, 2011.

\bibitem{sakurama2017communication}
K.~Sakurama and M.~Miura, ``Communication-based decentralized demand response
  for smart microgrids,'' \emph{IEEE Transactions on Industrial Electronics},
  vol.~64, no.~6, pp. 5192--5202, 2017.

\bibitem{tan2014optimal}
Z.~Tan, P.~Yang, and A.~Nehorai, ``An optimal and distributed demand response
  strategy with electric vehicles in the smart grid,'' \emph{IEEE Transactions
  on Smart Grid}, vol.~5, no.~2, pp. 861--869, 2014.

\bibitem{chang2012coordinated}
T.-H. Chang, M.~Alizadeh, and A.~Scaglione, ``Coordinated home energy
  management for real-time power balancing,'' \emph{IEEE Power and Energy
  Society General Meeting}, pp. 1--8, 2012.

\bibitem{rotering2011optimal}
N.~Rotering and M.~Ilic, ``Optimal charge control of plug-in hybrid electric
  vehicles in deregulated electricity markets,'' \emph{IEEE Transactions on
  Power Systems}, vol.~26, no.~3, pp. 1021--1029, 2011.

\bibitem{foster2013optimal}
J.~M. Foster and M.~C. Caramanis, ``Optimal power market participation of
  plug-in electric vehicles pooled by distribution feeder,'' \emph{IEEE
  Transactions on power systems}, vol.~28, no.~3, pp. 2065--2076, 2013.

\bibitem{koutsopoulos2012optimal}
I.~Koutsopoulos and L.~Tassiulas, ``Optimal control policies for power demand
  scheduling in the smart grid,'' \emph{IEEE Journal on Selected Areas in
  Communications}, vol.~30, no.~6, pp. 1049--1060, 2012.

\bibitem{meyn2015ancillary}
S.~P. Meyn, P.~Barooah, A.~Bu{\v{s}}i{\'c}, Y.~Chen, and J.~Ehren, ``Ancillary
  service to the grid using intelligent deferrable loads,'' \emph{IEEE
  Transactions on Automatic Control}, vol.~60, no.~11, pp. 2847--2862, 2015.

\bibitem{Bernstein2002complexity}
D.~S. Bernstein, R.~Givan, N.~Immerman, and S.~Zilberstein, ``The complexity of
  decentralized control of {M}arkov decision processes,'' \emph{INFORMS,
  Mathematics of operations research}, vol.~27, no.~4, pp. 819--840, Nov. 2002.

\bibitem{arabneydi2016new}
J.~Arabneydi, ``New concepts in team theory: Mean field teams and reinforcement
  learning,'' Ph.D. dissertation, McGill University, 2016.

\bibitem{JalalCDC2014}
J.~Arabneydi and A.~Mahajan, ``Team optimal control of coupled subsystems with
  mean-field sharing,'' \emph{in Proceedings of 53rd IEEE Conference on
  Decision and Control}, pp. 1669--1674, 2014.

\bibitem{JalalCDC2017}
J.~Arabneydi and A.~G. Aghdam, ``A certainty equivalence result in team-optimal
  control of mean-field coupled {M}arkov chains,'' \emph{in Proceedings of the
  \nth{56} IEEE Conference on Decision and Control}, pp. 3125--3130, 2017.

\bibitem{Olfati2006belief}
R.~Olfati-Saber, E.~Franco, E.~Frazzoli, and J.~S. Shamma, ``Belief consensus
  and distributed hypothesis testing in sensor networks,'' \emph{Springer,
  Networked Embedded Sensing and Control}, vol. 331, pp. 169--182, Jul. 2006.

\bibitem{Bishop2014}
A.~N. Bishop and A.~Doucet, ``Distributed nonlinear consensus in the space of
  probability measures,'' \emph{Elsevier, 19th IFAC Proceedings Volumes},
  vol.~47, no.~3, pp. 8662 -- 8668, 2014.

\bibitem{grinstead2012introduction}
C.~M. Grinstead and J.~L. Snell, \emph{Introduction to probability}.\hskip 1em
  plus 0.5em minus 0.4em\relax American Mathematical Society, 2012.

\bibitem{kumar2015stochastic}
P.~R. Kumar and P.~Varaiya, \emph{Stochastic systems: Estimation,
  identification, and adaptive control}.\hskip 1em plus 0.5em minus 0.4em\relax
  SIAM (Classics in Applied Mathematics), 2015, vol.~75.

\bibitem{Bertsekas2012book}
D.~P. Bertsekas, \emph{Dynamic programming and optimal control}.\hskip 1em plus
  0.5em minus 0.4em\relax Athena Scientific, 2012.

\end{thebibliography}
\bibliographystyle{IEEEtran}
 
\end{document}